\documentclass{article}[12pt]

\usepackage{amsmath,amssymb,amsthm}
\usepackage[affil-it]{authblk}
\usepackage{enumitem}

\makeindex

\newcommand{\sm}{\setminus}
\newcommand{\su}{\subseteq}
\newcommand{\Z}{\mathbb Z}

\newcommand{\1}{\mathbf 1}

\newtheorem{theorem}{Theorem}[section]
\newtheorem{lemma}[theorem]{Lemma}
\newtheorem{corollary}[theorem]{Corollary}
\newtheorem{proposition}[theorem]{Proposition}

\theoremstyle{definition}

\newtheorem{remark}[theorem]{Remark}
\newtheorem{observation}[theorem]{Observation}

\author{Krishnendu Paul, Shameek Paul%
\thanks{E-mail addresses: \texttt{krishnendupaul@ggdcgopi2.ac.in, shameek.paul@rkmvu.ac.in}}}
\affil{\small
Government General Degree College Gopiballavpur-II,
P.O. Beliaberah, Dist. Jhargram, 721517, India \\

\bigskip
Ramakrishna Mission Vivekananda Educational and Research Institute, P.O. Belur Math, Dist. Howrah, 711202, India}

\date{}
\title {$\{\pm 1\}$-weighted zero-sum constants}

\begin{document}

\baselineskip=14.5pt

\maketitle

\begin{abstract}
Let $A,B\su \Z_n\sm\{0\}$. A sequence $S=(x_1,\ldots, x_k)$ in $\Z_n$ is called an $(A,B)$-weighted zero-sum sequence if there exist $a_1,\ldots,a_k\in A$ and $b_1,\ldots,b_k\in B$ such that $a_1x_1+\cdots+a_kx_k=0$ and $b_1a_1+\cdots+b_ka_k=0$. The constant $E_{A,B}(n)$ is defined to be the smallest positive integer $k$ such that every sequence of length $k$ in $\Z_n$ has an $(A,B)$-weighted zero-sum subsequence of length $n$. We determine the constant $E_{A,B}(n)$ and the related constants $C_{A,B}(n)$ and $D_{A,B}(n)$ when $A=\{\pm 1\}$ and $B=\{1\}$.
\end{abstract}

\section{Introduction}\label{intro}

By a sequence $S$ in a set $X$ of length $k$, we mean an element of the set $X^k$. By a subsequence we shall always mean a nonempty subsequence. Throughout this article, $M$ will denote a finite module over the given ring. The following definition is given in \cite{KS}.

Let $R$ be a ring and let $A,B$ be nonempty subsets of $R\sm \{0\}$. A sequence $S=(x_1,\ldots,x_k)$ in $M$ is called an $(A,B)$-weighted zero-sum sequence if there exist $a_1,\ldots,a_k\in A$ and $b_1,\ldots,b_k\in B$ such that $a_1x_1+\cdots +a_kx_k=0$ and $b_1a_1+\cdots +b_ka_k=0$.

The constant $D_{A,B}(M)$ is the smallest positive integer $k$ such that every sequence having length $k$ in $M$ has an $(A,B)$-weighted zero-sum subsequence. The constant $C_{A,B}(M)$ is the least positive integer $k$ such that every sequence having length $k$ in $M$ has an $(A,B)$-weighted zero-sum subsequence of consecutive terms.

The constant $E_{A,B}(M)$ is the least positive integer $k$ such that every sequence having length $k$ in $M$ has an $(A,B)$-weighted zero-sum subsequence having length $|M|$.

We let $D_A(M)$ denote the constant $D_{A,B}(M)$ where $B=\{0\}$. We define the constants $C_A(M)$ and $E_A(M)$ in a similar manner. Also, we let $D(M)$ denote the constant $D_\1(M)$ and we define $C(M)$ and $E(M)$ in a similar manner.

For every $n\in \mathbb N$ with $n\geq 2$ we denote the ring $\mathbb Z/n\mathbb Z$ by $\mathbb Z_n$. In the case $M=R=\Z_n$, we replace the symbol $M$ by $n$ in the notation for the above constants.

Let $S=(x_1,\ldots,x_k)$ be a sequence in $M$ and let $x\in M$. We define the {\it translate} of $S$ by $x$ to be the sequence $S+x=(x_1+x,\ldots,x_k+x)$. We denote the set $\{1\}$ by $\1$.

\begin{observation}\label{tr}
Suppose a sequence $S$ in a module $M$ is an $(A,\1)$-weighted zero-sum sequence. Then every translate of $S$ is also an $(A,\1)$-weighted zero-sum sequence, since the following identity holds.
\[a_1(x_1+x)+\cdots+a_k(x_k+x)=a_1x_1+\cdots +a_kx_k+(a_1+\cdots+a_k)x.\]
Hence, to show that a sequence $S$ has an $(A,\1)$-weighted zero-sum subsequence, it is enough to show that this is true for a translate of $S$.
\end{observation}

In this article, we show the following results:

\noindent
Let $M=R=\Z_n$ and $A=\{\pm 1\}$. Then
\begin{itemize}
\item We have $C_{A,\1}(n)=2C_A(n)$ and $D_A(n)+1\leq D_{A,\1}(n)\leq 2D_A(n)$.

\item When $n$ is odd, we have $E_{A,\1}(n)=2n-1$.

\item When $n$ is even, we have $E_A(n)\leq E_{A,\1}(n)\leq n-2+D_{A,\1}(n)$.
\end{itemize}

\noindent
Let $M$ be a finite $\Z_2$-module. Then $\{\pm 1\}=\{1\}$ and we have
\[C_{\1,\1}(M)=2C(M),~D_{\1,\1}(M)=D(M)+1,\text{ and }E_{\1,\1}(M)=E(M).\]
From this, we can show that if $M=\Z_2^r$, then
\[C_{\1,\1}(M)=2^{r+1},~D_{\1,\1}(M)=r+2,\text{ and }E_{\1,\1}(M)=2^r+r.\]

\section{Some lower bounds}

Given a ring $R$, we let $R^*$ denote the group of units of the ring $R$.

\begin{observation}\label{al2}
Let $T$ be an $(A,B)$-weighted zero-sum sequence in $M$. If either $A\su R^*$ or $B\su R^*$, then $T$ has length at least two, since $0\notin A\cup B$.
\end{observation}

\begin{proposition}\label{clb}
If either $A\su R^*$ or $B\su R^*$, then $C_{A,B}(M)\geq 2C_A(M)$.
\end{proposition}

\begin{proof}
Let $k=C_A(M)$ and let $S'=(x_1,\ldots,x_{k-1})$ be a sequence in $M$ which does not have any $A$-weighted zero-sum subsequence of consecutive terms. Let $S=(0,x_1,0,x_2,0,\ldots,0,x_{k-1},0)$ be the sequence which is obtained from $S'$ by putting zeroes in alternate positions. Suppose $S$ has an $(A,B)$-weighted zero-sum subsequence of consecutive terms.  Then by Observation \ref{al2} we get the contradiction that $S'$ has an $A$-weighted zero-sum subsequence of consecutive terms. Since $S$ has length $2k-1$, it follows that $C_{A,B}(M)\geq 2k$.
\end{proof}

\begin{proposition}\label{dlb}
If either $A\su R^*$ or $B\su R^*$, then $D_{A,B}(M)\geq D_A(M)+1$.
\end{proposition}

\begin{proof}
Let $k=D_A(M)$ and let $S'=(x_1,\ldots,x_{k-1})$ be a sequence in $M$ which does not have any $A$-weighted zero-sum subsequence. Let $S=(x_1,\ldots,x_{k-1},0)$. Suppose $S$ has an $(A,B)$-weighted zero-sum subsequence. By Observation \ref{al2} we get the contradiction that $S'$ has an $A$-weighted zero-sum subsequence. Since $S$ has length $k$, it follows that $D_{A,B}(M)\geq k+1$.
\end{proof}

\section{Upper bounds when $A=\{\pm 1\}$ and $B=\{1\}$}

For $a,b\in\mathbb Z$ we denote the set $\{x\in\mathbb Z:a\leq x\leq b\}$ by $[a,b]$. We let $|X|$ denote the number of elements in a finite set $X$.
For sequences $T$ and $T'$ in $M$, we let $T+T'$ denote the concatenation of $T$ and $T'$. For a sequence $S=(x_1,\ldots,x_k)$ in $M$, we say that the sum of $S$ is $\sum S=x_1+\cdots+x_k$.

\begin{observation}\label{ssl}
Let $A=\{\pm 1\}$ and let $S$ be a sequence in $M$. Suppose there exist two disjoint subsequences of $S$, say $T$ and $T'$, such that $S$ is a permutation of $T+T'$, and both the sequences $T$ and $T'$ have the same sum and length. Then $S$ is an $(A,\1)$-weighted zero-sum sequence.
\end{observation}

\begin{theorem}\label{cub}
Let $A=\{\pm 1\}$. We have $C_{A,\1}(M)=2C_A(M)$.
\end{theorem}

\begin{proof}
Let $k=C_A(M)$ and let $S=(x_1,\ldots,x_{2k})$ be a sequence in $M$. For every $i\in [1,k]$ let $y_i=x_{2i-1}-x_{2i}$. Consider the sequence $S'=(y_1,\ldots,y_k)$. Since $C_A(M)=k$, we see that $S'$ has an $A$-weighted zero-sum subsequence $T'$ of consecutive terms. Thus, we see that there exist $i,j\in [1,k]$ such that $i\leq j$ and $T'=(y_i,\ldots,y_j)$. Let $T=(x_{2i-1},x_{2i},\ldots,x_{2j-1},x_{2j})$. Since $A=\{\pm 1\}$, we see that $T$ is an $(A,\1)$-weighted zero-sum subsequence of consecutive terms of $S$. Hence, it follows that $C_{A,\1}(M)\leq 2k$. Since $A\su R^*$, we are done by Proposition \ref{clb}.
\end{proof}

We give a combinatorial proof of the next result using lattice paths.

\begin{lemma}\label{binom}
We have $\binom{2k}{k}>2^k$.
\end{lemma}

\begin{proof}
Let $D=\big\{(x,y)\in \Z\times \Z:x+y=k\big\}$, let $\mathcal L$ be the set of all lattice paths from $(0,0)$ to $(k,k)$, and let $\mathcal L'$ be the set of all lattice paths from $(0,0)$ to a point of $D$. There are $\binom{2k}{k}$ paths in $\mathcal L$, since the paths in $\mathcal L$ correspond to binary strings of length $2k$ which have weight $k$. There are $2^k$ paths in $\mathcal L'$, since the paths in $\mathcal L'$ correspond to binary strings of length $k$. Each path in $\mathcal L$ passes through a unique point of $D$. So we get an onto map from $\mathcal L$ to $\mathcal L'$. Further, there exist at least two paths in $\mathcal L$ which are mapped to the path in $\mathcal L'$ which corresponds to a weight one binary string having length $k$.
\end{proof}

\begin{theorem}\label{ubd}
Let $A=\{\pm 1\}$ and let $k$ be such that $2^k\geq |M|$. Then we have $D_{A,\1}(M)\leq 2k$.
\end{theorem}

\begin{proof}
Let $S=(x_1,\ldots,x_{2k})$ be a sequence in $M$. By Lemma \ref{binom} we see that $\binom{2k}{k}>2^k$. Since $2^k\geq |M|$, it follows that $\binom{2k}{k}>|M|$. Thus, there exist two distinct subsequences $T'$ and $T''$ of $S$ having length $k$ such that both $T'$ and $T''$ have the same sum. Let $T$ be the subsequence of $S$ whose terms are the common terms of $T'$ and $T''$. Since the subsequences $T'-T$ and $T''-T$ are disjoint and they have the same sum and length, by Observation \ref{ssl} we see that $S$ has an $(A,\1)$-weighted zero-sum subsequence. It follows that $D_{A,\1}(M)\leq 2k$.
\end{proof}

\begin{observation}\label{char}
Let $A=\{\pm 1\}$. Suppose $\text{char }R$ is $n$. When $n$ is even, every $(A,\1)$-weighted zero-sum sequence in $M$ has even length. When $n$ is odd, every $(A,\1)$-weighted zero-sum sequence in $M$ having length $n$ is a zero-sum sequence.
\end{observation}

\begin{theorem}\label{even}
Let $\emph{char }R$ be even, let $m$ be an even number, and let $A=\{\pm 1\}$. Then every sequence in $M$ of length $m-2+D_{A,\1}(M)$ has an $(A,\1)$-weighted zero-sum subsequence of length $m$.
\end{theorem}

\begin{proof}
Let $S$ be a sequence in $M$ of length $m-2+d$ where $d=D_{A,\1}(M)$. Let $r$ be the largest integer such that $S$ contains $r$ pairs of repeated terms. Let $T$ be the subsequence of $S$ whose terms are these $r$ pairs of repeated terms and let $S'=S-T$.

Suppose $T$ has at least $m$ terms. In this case, we can find an $(A,\1)$-weighted zero-sum subsequence of $S$ having length $m$ by taking $m/2$ pairs of repeated terms from $T$. Since the length of $T$ is even, we may assume that $T$ has at most $m-2$ terms. It follows that $S'$ has at least $d$ terms.

Let $T'$ be a maximal $(A,\1)$-weighted zero-sum subsequence of $S'$. Let the length of $T'$ be $l$. By Observation \ref{char} we see that $l$ is even. Since $S'-T'$ has at most $d-1$ terms and since $S'-T'=(S-T)-T'=S-(T+T')$, we see that $T+T'$ has at least $m-1$ terms.

It follows that $T$ has at least $m-l-1$ terms. Since the length of $T$ is even, we see that $T$ has at least $m-l$ terms. Since $m-l$ is even, we may add $(m-l)/2$ pairs of repeated terms from $T$ to the subsequence $T'$ to get an $(A,\1)$-weighted zero-sum subsequence of $S$ having length $l+(m-l)=m$.
\end{proof}

\begin{corollary}\label{eam}
Let $A=\{\pm 1\}$. Suppose $\emph{char }R$ is even and $m=|M|$ is even. Then we have $E_A(M)\leq E_{A,\1}(M)\leq m-2+D_{A,\1}(M)$.
In particular, it follows that if $D_{A,\1}(M)=D_A(M)+1$, then $E_{A,\1}(M)=E_A(M)$.
\end{corollary}

\begin{proof}
Clearly, we have $E_A(M)\leq E_{A,\1}(M)$. Also, by Theorem \ref{even} we see that $E_{A,\1}(M)\leq m-2+D_{A,\1}(M)$. Thus, if $D_{A,\1}(M)=D_A(M)+1$, we see that $E_{A,\1}(M)\leq m-1+D_A(M)$. It is easy to see that $m-1+D_A(M)\leq E_A(M)$. Hence, it follows that $E_{A,\1}(M)=E_A(M)$.
\end{proof}

\section{$R=\Z_2$ and $A=\{1\}$}

\begin{observation}\label{211}
Let $R=\Z_2$ and let $S$ be a sequence in $M$ having even length. If $S$ is a zero-sum sequence, then $S$ is a $(\1,\1)$-weighted zero-sum sequence.
\end{observation}

\begin{theorem}\label{d11}
Let $R=\Z_2$. Then $C_{\1,\1}(M)=2C(M)$, $D_{\1,\1}(M)=D(M)+1$, and $E_{\1,\1}(M)=E(M)$.
\end{theorem}

\begin{proof}
Let $S$ be a sequence in $M$ having length $D(M)+1$. Then the sequence $S$ has a zero-sum subsequence $T$. Let $x$ be a term of $T$ and let $S'=S-(x)$. Since $S'$ has length $D(M)$, we see that $S'$ has a zero-sum subsequence $T'$. Since $x$ is a term of $T$ which is not a term of $T'$, we see that $T\neq T'$.

Suppose both $T$ and $T'$ have odd length. Let $U$ be the subsequence of $S$ whose terms are those terms of $T$ and $T'$ which are not terms of both $T$ and $T'$. Since the length of $U$ has the same parity as the sum of the lengths of $T$ and $T'$, we see that $U$ has even length. Since $R=\Z_2$, we see that $\sum U=\sum T+\sum T'$.

Thus, $U$ is a zero-sum sequence. Hence, by Observation \ref{211} we see that $S$ has a $(\1,\1)$-weighted zero-sum subsequence. It follows that $D_{\1,\1}(M) \leq D(M)+1$. Hence, by using Proposition \ref{dlb} we see that $D_{\1,\1}(M)=D(M)+1$. So from Corollary \ref{eam} we see that $E_{\1,\1}(M)=E(M)$. Also, from Theorem \ref{cub} we see that $C_{\1,\1}(M)=2C(M)$.
\end{proof}

Let $M$ be a finite $R$-module. Since $\Z_2$ is a field, we see that $M=\Z_2^r$ where $r$ is the dimension of $M$. From Theorems 3.1, 4.1, and 5.2 of \cite{SKS2} we can deduce that $C(M)=2^r$, $D(M)=r+1$, and $E(M)=2^r+r$. Hence, from Theorem \ref{d11} it follows that $C_{\1,\1}(M)=2^{r+1}$, $D_{\1,\1}(M)=r+2$, and $E_{\1,\1}(M)=2^r+r$.

\section{$M=R=\Z_n$ and $A=\{\pm 1\}$}

From \cite{ACFKP} we see that if $D_A(n)=k$, then $2^k>n$. Hence, from Proposition \ref{dlb} and Theorem \ref{ubd}, we see that $D_A(n)+1\leq D_{A,\1}(n)\leq 2D_A(n)$.
When $n$ is odd, by Observation \ref{char} it follows that $E_{A,\1}(n)=E(n)=2n-1$. When $n$ is even,  Corollary \ref{eam} shows $E_A(n)\leq E_{A,\1}(n)\leq n-2+D_{A,\1}(n)$.
In particular, if $D_{A,\1}(n)=D_A(n)+1$, then $E_{A,\1}(n)=E_A(n)$. Also, from Theorem \ref{cub} we see that $C_{A,\1}(n)=2C_A(n)$.

\begin{remark}
We can check that the sequence $S=(0,1,2,4)$ in $\Z_6$ has no $(A,\1)$-weighted zero-sum subsequence. It follows that $D_{A,\1}(6)\geq 5$. Suppose $S$ is a sequence having length five in $\Z_6$. Since $\binom{5}{2}=10>6$, we can find two subsequences of $S$ having length two which have the same sum. It follows that $S$ has an $(A,\1)$-weighted zero-sum subsequence. Hence, we see that $D_{A,\1}(6)=5$.

From \cite{ACFKP} we see that $D_A(6)=3$. Thus, we have $D_{A,\1}(6)=D_A(6)+2$. By considering the sequence $(0,1,2)$ in $\Z_4$ and the sequence $(0,1,2,4)$ in $\Z_8$ and by using a similar argument as in the previous paragraph, we can show that $D_{A,\1}(4)=D_A(4)+1$ and $D_{A,\1}(8)=D_A(8)+1$.
\end{remark}

\section{Concluding remarks}

Let $M=R=\Z_n$ and $A=\{\pm 1\}$. In \cite{SKS1} it is shown that $C_A(n)=n$ when $n$ is a power of $2$. As per our knowledge, the values of $C_A(n)$ are not known for other $n$ (except for some small values of $n$). So from Theorem \ref{cub} we see that $C_{A,\1}(n)=2n$ when $n$ is a power of $2$.

We have not found any value of $n$ for which $D_{A,\1}(n)\geq D_A(n)+3$. We feel that $D_{A,\1}(n)=D_A(n)+1$ whenever $n$ is a power of $2$. We also feel that $E_{A,\1}(n)=E_A(n)$ for every even number $n$.

Let $|M|=m$ and let $m$ be a multiple of char $R$. In \cite{KS} it is shown that $D_{A,B}(M)\leq E_{A,B}(M)\leq 2m-1$ and that $C_{A,B}(M)\leq m^2$.
Also, in \cite{KS} it is shown that $D_{\1,\1}(n)=E_{\1,\1}(n)=2n-1$ and $C_{\1,\1}(n)=n^2$. When $A=\{\pm 1\}$, we have obtained much smaller upper bounds for the corresponding $(A,\1)$-weighted zero-sum constants.

\end{document}